\documentclass[12pt]{amsart}

\usepackage{amsmath, amssymb, amsthm, amsfonts, mathrsfs}
\usepackage{cite}
\usepackage{url}
\usepackage{graphicx}
\usepackage{caption}
\usepackage{fullpage}

\newtheorem{theorem}{Theorem}
\newtheorem*{theorem*}{Theorem}

\theoremstyle{definition}

\newcommand{\ess}[1]{\hat{#1}}
\newcommand{\incentive}{\varphi}

\makeatletter
\g@addto@macro{\endabstract}{\@setabstract}
\newcommand{\authorfootnotes}{\renewcommand\thefootnote{\@fnsymbol\c@footnote}}%
\makeatother
\begin{document}

\title{Mean Evolutionary Dynamics for Stochastically Switching Environments}
\author{Marc Harper}
\address[Marc Harper]{Department of Genomics and Proteomics, UCLA}
\email[Marc Harper]{marcharper@ucla.edu (corresponding author)}
\author{Dashiell Fryer}
\author{Andrew Vlasic}

\begin{center}
  \LARGE 
  Mean Evolutionary Dynamics for Stochastically Switching Environments \par \bigskip

  \normalsize
  \authorfootnotes
  Marc Harper\footnote{corresponding author, email: marcharper@ucla.edu}\textsuperscript{1}, Dashiell Fryer\textsuperscript{2},
  Andrew Vlasic\textsuperscript{3} \par \bigskip

  \textsuperscript{1}Department of Genomics and Proteomics, UCLA \par
  \textsuperscript{2}Department of Mathematics, Pomona College\par
  \textsuperscript{3}Department of Mathematics, Indiana University East\par \bigskip
  
\end{center}


\begin{abstract}
Populations of replicating entities frequently experience sudden or cyclical changes in environment. We explore the implications of this phenomenon via a environmental switching parameter in several common evolutionary dynamics models including the replicator dynamic for linear symmetric and asymmetric landscapes, the Moran process, and incentive dynamics. We give a simple relationship between the probability of environmental switching, the relative fitness gain, and the effect on long term behavior in terms of  fixation probabilities and long term outcomes for deterministic dynamics. We also discuss cases where the dynamic changes, for instance a population evolving under a replicator dynamic switching to a best-reply dynamic and vice-versa, giving Lyapunov stability results.
\end{abstract}

\section{Introduction}

Many replicating organisms live in periodically or suddenly changing selective landscapes. It may be easier to list non-examples than to attempt to describe the many biological scenarios for which this phenomenon holds; we describe a few for motivation. Intermittent or frequent flooding is a very simple example of a stochastically switching environment. Such organisms may behave very differently depending on the environment, for instance in the presence of water, such as the flowering of desert plants and the reproductive cycle of some amphibians. Sometimes the organism can survive in both environments yet can only reproduce in one environment, or the organism reproduces at very different rates in the two environments. An adaptive advantage of a subpopulation may fixate quickly or not at all depending on the relative advantage, the cost of the variant behavior or trait, and the frequency at which the environments switch.

Sickle-cell anemia is a well-known instance of heterozygote advantage in humans \cite{allison1954protection} \cite{aidoo2002protective}, which we frame simply in the following way. During \emph{normal} environmental conditions, having the gene for sickle-cell anemia on both chromosomes (homozygous for sickle-cell) is maladaptive, and simply being a carrier (heterozygous) is also of lower fitness because some offspring may exhibit the disease. During malaria outbreaks, however, sickle-cell carriers have a fitness advantage due to conferred malaria resistance. Thus there is selection pressure against the extinction of malarial-resistance alleles in places where malaria outbreaks are common, and we can model this as a population of two types (sickle-cell carriers and wild-type) with two different selective landscapes that alter the order of fitness for the two types. Call these two landscapes $E_1$, corresponding to the non-malarial environment, and $E_2$, when malaria is present in the population, food-supply, or when the population is otherwise exposed to malaria. We will assume that the population experiences environment $E_1$ with probability $p$ and environment $E_2$ with probability $1-p$.

The social amoeba Dictyostelium discoideum has a complex life-cycle \cite{loomis1982development} \cite{devreotes1989dictyostelium} \cite{eichinger2005genome}, which we simplistically model as follows. For part of the life cycle, Dictyostelium cells behave independently as single-celled organisms, foraging for food. As food becomes scarce, some of the single-cells come back together to form a spore-dispersing fruiting-body. We model this as a stochastic environment corresponding to the presence of food or not, and two population types (cooperators and defectors). The $E_1$ landscape is the independent single-celled case, where defectors have the advantage, since cooperating together would lessen the total area in which the population forages, concentrating the subpopulation of cooperators in a smaller area (thus making less food available to cooperators on average). The $E_2$ environment favors cooperators, since as food becomes scarce, the cells that participate in the fruiting body have a larger change of reproducing and surviving in the longer term than the individuals that continue to forage locally. $E_1$ is a prisoner's dilemma favoring defectors and $E_2$ could be modeled as a constant relative fitness landscape with $r > 1$ for cooperators (the fitness for defectors is non-zero since the defectors could stumble upon a new food source).

We can also interpret the proposed model as giving a relationship between the evolution and fixation of traits that respond to rare events and environments in terms of the difference in fitnesses conferred by the trait and probability of the rare environment. Intuitively, a sufficiently rare event may not provide enough \emph{exposure} for a relevant trait to be fixate, such as an organism being crushed by a meteor. This event is likely too rare and too extreme for a trait defending against it to arise, and even if such a trait arises by chance, it is unlikely to persist selectively because of the rarity of the event. More generally, if the cost of the responsive trait is too large and the alternate environment is not common enough, such a trait may fail to persist in the population.

This simple model is in contrast to a model in which the fitness landscape is explicitly time-dependent, such as temperature variation in different seasons, in which the environment continuously cycles through different extremes. Although such a model may be more accurate, it poses substantial analytic difficulties. Henceforth we will not consider explicit examples of the above scenarios or others; rather we pursue models that could apply to these scenarios and many others in the context of evolutionary dynamics. We present our models in terms of a deterministic replicator dynamic \cite{hofbauer2003evolutionary} \cite{nowak2006evolutionary} \cite{weibull1995evolutionary} \cite{hofbauer1998evolutionary} with symmetric and asymmetric landscapes, the Moran process \cite{moran1958random} \cite{moran1962statistical} \cite{taylor2004evolutionary} \cite{nowak2004emergence} \cite{fudenberg2004stochastic}, a probabilistic Markov process, and for the deterministic incentive dynamics \cite{fryer2012existence}. In the deterministic cases, we are essentially assuming infinitely large populations in which in any given time interval, the environments $E_1$ and $E_2$ are experienced for proportions $p$ and $1-p$ of the interval, respectively. A more realistic model would be that the population experiences each landscape continuously for intervals of relative length $p$ and $1-p$, however we find good agreement with the Moran process in terms of fixation probabilities \cite{antal2006fixation}, which are essentially averaged over all sequences of $E_1$ and $E_2$ appearing with probabilities $p$ and $1-p$. In other words, we primarily investigate mean dynamics rather than those in which the switching parameter is truly stochastic (e.g. a random variable). Such models may have significantly more interesting behaviors, but are much more difficult to approach analytically, and are often analyzed in the context of the average system \cite{belykh2011dynamics}. Hence our strategy is to first understand the mean dynamics to light the way for more stochastic approaches.

\section{Deterministic Replicator Dynamics for Symmetric Landscapes}

We first consider the implications of the replicator equation, a standard continuous and deterministic model of selection \cite{nowak2006evolutionary} \cite{taylor2004evolutionary}. The replicator equation takes a fitness landscape $f$ on $n$-types in an infinitely large population. The proportions of each type is denoted by $x_i$ for $1 \leq i \leq n$, and the proportions evolve in time as
\[ \dot{x}_i = x_i (f_i(x) - \bar{f}(x)),\]
where $\bar{f}(x) = \sum_i x_i f_x(x)$ is the mean fitness. A common fitness landscape is $f(x) = A x$, where $A$ is a game-matrix for the fitness payouts from interactions between the replicating individuals of each type.

\subsection{Constant Relative Fitness Landscapes}

Abstractly, consider a model where a population consists of two types of replicators, and a stochastic environmental occurrence (such as malaria outbreaks or frequent flooding). For a population of two replicating types 1 and 2, suppose during environment $E_1$, occurring with probability $p$, that type 1 has relative fitness $s$ and that during environment $E_2$, occurring with probability $1-p$, that type 1 has relative fitness $t$. Let $A_r$ denote the matrix 
\[ A_r = \left( \begin{matrix} r & r \\ 1 & 1 \end{matrix} \right), \] which indicates that type 1 has relative fitness advantage $r$ over type 2. The effective game matrix for the two environments over a long period of time is simply $A_r  = p A_s + (1-p) A_t$, and so $r = p s + (1-p) t$. 

For the replicator equation with game matrix $A_r$, either type $A$ or $B$ will fixate depending on whether $r > 1$ or not. In terms of the stochastic environment parameter $p$, $r > 1$ iff 
\begin{equation}
p (s - t) > (1-t) 
\label{balance_condition}
\end{equation}
So we see that fixation depends on the rarity of the environmental stochasticity and the relative fitnesses. The special case in which $s > t = 1$, meaning that both types have the same fitness during environment $E_2$ but type 1 has an advantage during environment $E_1$, then inequality (\ref{balance_condition}) reduces to $p>0$. In other words, type 1 will eventually dominate due to its ability to respond to environment $E_1$.

\subsection{Frequency-dependent Fitness Landscapes}

Ross Cressman gives a classification of the phase portrait types of the replicator dynamics for linear 2x2 games as follows \cite{cressman2003evolutionary}. For a game matrix
\[ A_{a,b,c,d} = \left( \begin{matrix} a & b \\ c & d \end{matrix} \right),\]
we have the four phase portraits described in Table \ref{two_by_two}.
\begin{figure}
\begin{tabular}{ccccc}
Name & Short Name & Conditions & Stable Equilibria & Unstable Equilibria\\
\hline
Prisoner's Dilemma & P1 & $a \leq c$, $d \geq b$ & $(1,0)$ & $(0,1)$\\
Prisoner's Dilemma & P2 & $a \geq c$, $d \leq b$ & $(0,1)$ & $(1,0)$\\
Hawk-Dove & HD & $a < c$, $d < b$ & $(\hat{x}, 1 - \hat{x})$ & $(0,1)$, $(1,0)$\\
Coordination & Co & $a>c$, $d>b$ &$(0,1)$, $(1,0)$ & $(\hat{x}, 1 - \hat{x})$
\end{tabular}
\caption{Non-degenerate $2 \times 2$ game behaviors (not both $a=c$ and $b=d$). Each has a distinct phase portrait and equilibria structure. The phase portraits of degenerate cases where $a=c$ and $d=b$ are stationary at all points. The internal rest point of the dynamic is $\hat{x} = (d-b) / (d-b + a-c)$.}
\label{two_by_two}
\end{figure}

Constant relative fitness landscapes are special cases of P1 and P2. We wish to know under what conditions can a stochastically switching environment alter the behavior of the replicator dynamic \emph{qualitatively}, i.e. in terms of the phase portrait. The linear combination of two arbitrary landscapes could have any particular portrait for various values of $p$. We start with a simple example of combining the two Prisoner's dilemma types. Let $A(p) = p A_{0,0,1,0} + (1-p) A_{0,1,0,0}$. Then we have that $A(0)$ is in class P1, $A(1)$ is class P2, and all other $p$ are HD, since $A(p) = A_{0,1-p,p,0}$ which has an internal stable equilibrium at $(p, 1-p)$. Similarly, if $A(p) = p A_{0,0,0,1} + (1-p) A_{1,0,0,0}$, then the boundary points $A(0)$ and $A(1)$ are of class P1 and P2 respectively and $A(p)$ is Co for all other $p$, again with internal equilibrium (but unstable) at $(p, 1-p)$. So we see that for some $E_1$ matrices, control over the stochastic alternate environment $E_2$ would allow the complete control over the dynamic outcomes of the combined system.

Any portrait can be altered to any other portrait with the right choice of alternate environment and probability of occurrence. Theorem \ref{perturbation_thm} shows that any $2 \times 2$ phase portrait can be perturbed by a stochastic landscape into any other phase portrait. That there are infinitely many such matrices is a result of the fact that adding a constant to any particular column of the game matrix does not alter the phase portrait \cite{hofbauer1998evolutionary}.

\begin{theorem}
Let $E_1, E_2 \in \{\text{P1}, \text{P2}, \text{HD}, \text{Co}\}$ and $p^{*} \in (0,1)$. For each matrix $A_1$ with elements such that $A_1$ is a game matrix for phase portrait $E_1$, there are infinitely many $A_2$ with phase portrait $E_2$ such that $A(p) = p A_1 + (1-p) A_2$ has portrait $E_1$ for $0 < p < p^{*}$ and has portrait $E_2$ for $p^{*} < p < 1$.
\label{perturbation_thm}
\end{theorem}
\begin{proof}
Since the signs of $a-c$ and $b-d$ are independent, it suffices to show just a single special case, namely that if $a_1 < c_1$ then there are $a_2, c_2 \in \mathbb{R}$ with $a_2 > c_2$ and the matrices $A(p)$ are as in the conclusion of the theorem. Given $A_1$ and $p^{*}$, let $b_2 = b_1$, $d_2 = d_1$, and note that $a-c = 0$ when $p = p^{*}$. Then $a_2$ can be chosen arbitrarily and $c_2 = a_2 - \frac{p^{*}}{1 - p^{*}} (c_1 - a_1)$.
\end{proof}

For the case of a HD to Co transition (or vice versa), the interior rest point moves to the boundary, changes stability for one instance, and moves back into the interior, in contrast to the examples given above. In this context, $A(p) = p A_1 + (1-p) A_2$ has the portrait class of $E_1$ if $A_2$ is any multiple of $A_{1,0,1,0}$ or $A_{0,1,0,1}$.

The classification of dynamics for $3 \times 3$ games is much more extensive \cite{bomze1983lotka} \cite{bomze1995lotka}. We consider only Rock-Scissors-Paper 3x3 games, with game matrix
\[ A_{a,b} = \left( \begin{matrix} 0 & a & -b \\ -b & 0 & a \\ a & -b & 0 \end{matrix} \right).\]
Similarly to the $2\times 2$ case, $p A_{a_1,b_1} + (1-p) A_{a_2, b_2} = A_{a', b'}$ where $a' = a_2 + p (a_1 - a_2)$ and $b' = b_2 + p (b_1 - b_2)$. For the RSP matrix there are essentially three outcomes: (1) $a = b \neq 0$: concentric orbits about the barycenter $(1/3, 1/3, 1/3)$; (2) $a > b$: divergence to the boundary; and (3) $a < b$: convergence to the barycenter. Hence if we have $a_1 = b_1$ and $a_2 < b_2$, then $a' - b' = (1-p)(a_2 - b_2) < 0$, so $a' < b'$ and the stochastic switching dynamic converges to the barycenter for any $p$ such that $0 < p < 1$. In other words, the stochastic environmental shifts are enough to break the concentric cycles and stabilize the population.

\section{Replicator Dynamics for Asymmetric Landscapes}

Now we consider examples where the fitness landscapes are described by separate matrices for each type. Given a population with two subpopulations $S_1$ and $S_2$, the payoffs for the interactions between them in the environment $E_1$ are
\[
\begin{array}{c|c|c|}
 \multicolumn{1}{c}{ } & \multicolumn{1}{c}{S_1 }  &  \multicolumn{1}{c}{S_2 } \\
\cline{2-3}
S_1 & a_{11},a_{11} & a_{12},a_{21} \\
\cline{2-3}
S_2 & a_{21},a_{12} & a_{22},a_{22}   \\
\cline{2-3}
\end{array}.
\]
Denote the payoff matrix as $A$. During environment $E_2$ the payoff for the types changes to 
\[ \begin{array}{c|c|c|}
 \multicolumn{1}{c}{ } & \multicolumn{1}{c}{S_1 }  &  \multicolumn{1}{c}{S_2 } \\
\cline{2-3}
S_1 & b_{11},b_{11} & b_{12},b_{21} \\
\cline{2-3}
S_2 & b_{21},b_{12} & b_{22},b_{22}   \\
\cline{2-3}
\end{array}, \]
which we denote the payoff matrix as $B$. As before we can write the expected game for the population as
\[ C = \left(
\begin{array}{cc}
 \big(1-p\big) a_{11} + pb_{11} &  \big(1-p\big) a_{12} + p b_{12} \\
\big(1-p\big) a_{21} + pb_{21} & \big(1-p\big) a_{22} + p b_{22}
\end{array} 
\right). \]

Consider the case when strategy 1 (type 1) is the only Nash equilibria for the payoff matrix $A$ (hence $S_1$ dominates), and strategy 2 (type 2) is the only Nash equilibria for the payoff matrix $B$ (hence $S_2$ dominates).  For the payoff matrix $C$, if $\displaystyle  c_{11} > c_{21}$, then a Nash equilibria exists at $\Big(  (1,0) , (1,0) \Big)$. Notice that $\displaystyle  c_{11} > c_{21}$  is equivalent to $\displaystyle \frac{ a_{11} -a_{21} }{  b_{21}-b_{11}  } > \frac{p}{1-p}$. Similarly, $C$ has a Nash Equilibria at $\Big(  (0,1) , (0,1) \Big)$ if $\displaystyle  \frac{p}{1-p} > \frac{  a_{12}-a_{22}  }{ b_{22} -b_{12} }$. If the inequalities $\displaystyle \frac{ a_{11} -a_{21} }{  b_{21}-b_{11}  } < \frac{p}{1-p}$ and $\displaystyle  \frac{p}{1-p} < \frac{  a_{12}-a_{22}  }{ b_{22} -b_{12} }$ hold, strategy 1 and strategy 2 are respectively dominated. 

\bigskip
\noindent From theses inequalities, we see that
\begin{enumerate}
\item Strategy 1 dominates if $\displaystyle \min \left\{  \frac{ a_{11} -a_{21} }{  b_{21}-b_{11}  } , \frac{  a_{12}-a_{22}  }{ b_{22} -b_{12} } \right\} > \frac{p}{1-p} $;
\bigskip
\item Strategy 2 dominates if $\displaystyle \max \left\{  \frac{ a_{11} -a_{21} }{  b_{21}-b_{11}  } , \frac{  a_{12}-a_{22}  }{ b_{22} -b_{12} } \right\} < \frac{p}{1-p}  $;
\bigskip
\item $C$ is a Co game if $\displaystyle \frac{ a_{11} -a_{21} }{  b_{21}-b_{11}  } > \frac{p}{1-p} >  \frac{  a_{12}-a_{22}  }{ b_{22} -b_{12} }$;
\bigskip
\item $C$ is a HD game if $\displaystyle \frac{ a_{11} -a_{21} }{  b_{21}-b_{11}  } < \frac{p}{1-p} <  \frac{  a_{12}-a_{22}  }{ b_{22} -b_{12} }$.
\end{enumerate}

Thus, if $p \ll 1$ or $1-p \ll 1$ then the population will fixate to $S_1$ or $S_2$, respectively. If the probability $p$ is not too large (or small), then the mixture gives rise to a HD or Co game depending on the inequality of the ratios of the differences. Notice for the coordination game, since the initial frequencies of the population would dictate evolution, a mutant subpopulation, though dominate during environment $E_2$, would not be able to invade. Generally, given two arbitrary payoff matrices $A$ and $B$, one can see that if a strategy is a Nash equilibria in both $A$ and $B$ then this strategy is also a Nash equilibria in $C$. Similarly, if a strategy is dominated in both $A$ and $B$ then this strategy is also dominated in $C$. An interesting consequence is when a strategy is a Nash equilibria in $A$, and the same strategy is dominated in $B$. Then there is a possibility for either characteristic for this strategy in $C$, given the proper parameter $p$.

In contrast to the symmetric case, particular choices of $A$ and $B$ can have different portraits for three intervals of nonzero length for $p \in [0,1]$. Consider the explicit example when 
\[A = \left(
\begin{array}{cc}
 5 &  4 \\
3 & 0
\end{array} 
\right)
\ \mbox{and} \
B= \left(
\begin{array}{cc}
-4 &  -12 \\
-1 & -6
\end{array} 
\right). \]
Calculating the ratios we see that $\displaystyle \frac{ a_{11} -a_{21} }{  b_{21}-b_{11} } =\frac{2}{3}$ and $\displaystyle \frac{  a_{12}-a_{22}  }{ b_{22} -b_{12} }=\frac{4}{5}$. If $ p<2/5$, the population will evolve to $S_1$; if $p>4/9$, the populace will evolve to $S_2$; and if $2/5<p<4/9$, the populace will evolve to a mixed population with the frequency $\displaystyle \left( \frac{3p}{1+2p}, \frac{1-p}{1+2p}  \right)$.

\section{Moran Process}

This replicator dynamics-based model assumes both a large population and that there are many environment switches such that they can be modeled as an average of $E_1$ $p$ of the time and $E_2$ $1-p$ of the time. We present an alternate model lacking both of these assumptions based on the Moran process for constant relative fitness landscapes. Consider a population as above of size $N$, with $i$ individuals of type $1$ and $N-i$ individuals of type $2$. Flip a $p$-biased coin, choosing $E_1$ with probability $p$ and $E_2$ with probability $1-p$. Then choose the appropriate game matrix corresponding to either $E_1$ or $E_2$ and proceed with the Moran process associated to this game matrix for one iteration, choosing an individual to reproduce proportionally to fitness and an individual to be replaced at random. The transition probabilities are:

\begin{align*}
T_{i \to i+1} &= \frac{i f_i}{i f_i + (N-i) f_2} \frac{N-i}{N} \\
T_{i \to i-1} &= \frac{(N-i) f_2}{i f_1 + (N-i) f_2} \frac{i}{N}, \\
\end{align*}
with $T_{i \to i} = 1 - T_{i \to i+1} - T_{i \to i-1}$. The fitness landscape is given by
\begin{align*}
f_1(i) &= \frac{a(i-1) + b(N-i)}{N-1} \\
f_2(i) &= \frac{ci + d(N-i-1)}{N-1}
\end{align*}
for a game matrix defined by
\[ A = \left( \begin{matrix}
 a & b\\
 c & d
\end{matrix} \right) \]
The classical Moran process is given by $a=r=b, c=1=d$. The process has two absorbing states $i=0$ and $i=N$, corresponding to the fixation of one of the two types. The sequence of coin flips will have a substantial impact on the eventual outcome of the process (i.e. to which type the population fixates). First let us consider the \emph{expected} outcome, averaged over the possible sequence of environmental states. Again suppose that we can model the population with an effective game matrix $A_r  = p A_s + (1-p) A_t$. The fixation probability of type 1 starting from population state $(i, N-i)$ is 
\begin{equation}
\phi = \frac{1 - r^{-i}}{1 - r^{-N}}
\label{moran_fixation_probability} 
\end{equation}
where $r = p s + (1-p) t$ as before. Consider the probability that a single mutant adapted to environment $E_1$ will fixate. The fixation probability is:
\[\phi = \frac{1 - r^{-1}}{1 - r^{-N}}\]
with $r = p(s-t) + t$. Consider a large population $N$ with $r > 1$ so that $r^{-N}$ is negligible. Then $\phi = 1 - 1/(p(s-t)+t)$, and depends on the value of $p$, $t$, and the difference in relative fitnesses $s-t$. As expected, larger values of $p$ and the fitness difference $s-t$ lead to larger values of $\phi$. Moreover, we see that if type 1 has a very large advantage during a very rare event, fixation may be very likely. Type 1 may proliferate during environment $E_2$ simply due to drift during the neutral fitness landscape, and when the rare environment $E_1$ occurs, type 1 dominates the fitness proportionate selection events. 

In general the same reasoning applies, tempered by the population size $N$. For large $N$, fixation becomes likely ($\phi > 1/2$) when $r > 2$, that is, when $p > (2-t) / (s-t)$, assuming $s > t$. Of course, variant types may persist for a long period of time before either type fixates. Compare this to the derivation above showing that the deterministic dynamic fixates the variant type if $p > (1-t) / (s-t)$, and note that $(2-t) / (s-t) = 1/ (s-t) + (1-t) / (s-t)$. Depending on the values of $t$ and $s$, the Moran process can predicts likely elimination even in cases where the deterministic dynamic fixates the mutant! For example, if $s=3$ and $t=1/2$, the the deterministic landscape fixates for $p > 2/5$ but the Moran process does not predict likely fixation until $p > 3/5$.

\begin{figure}
    \centering
    \includegraphics[width=\textwidth]{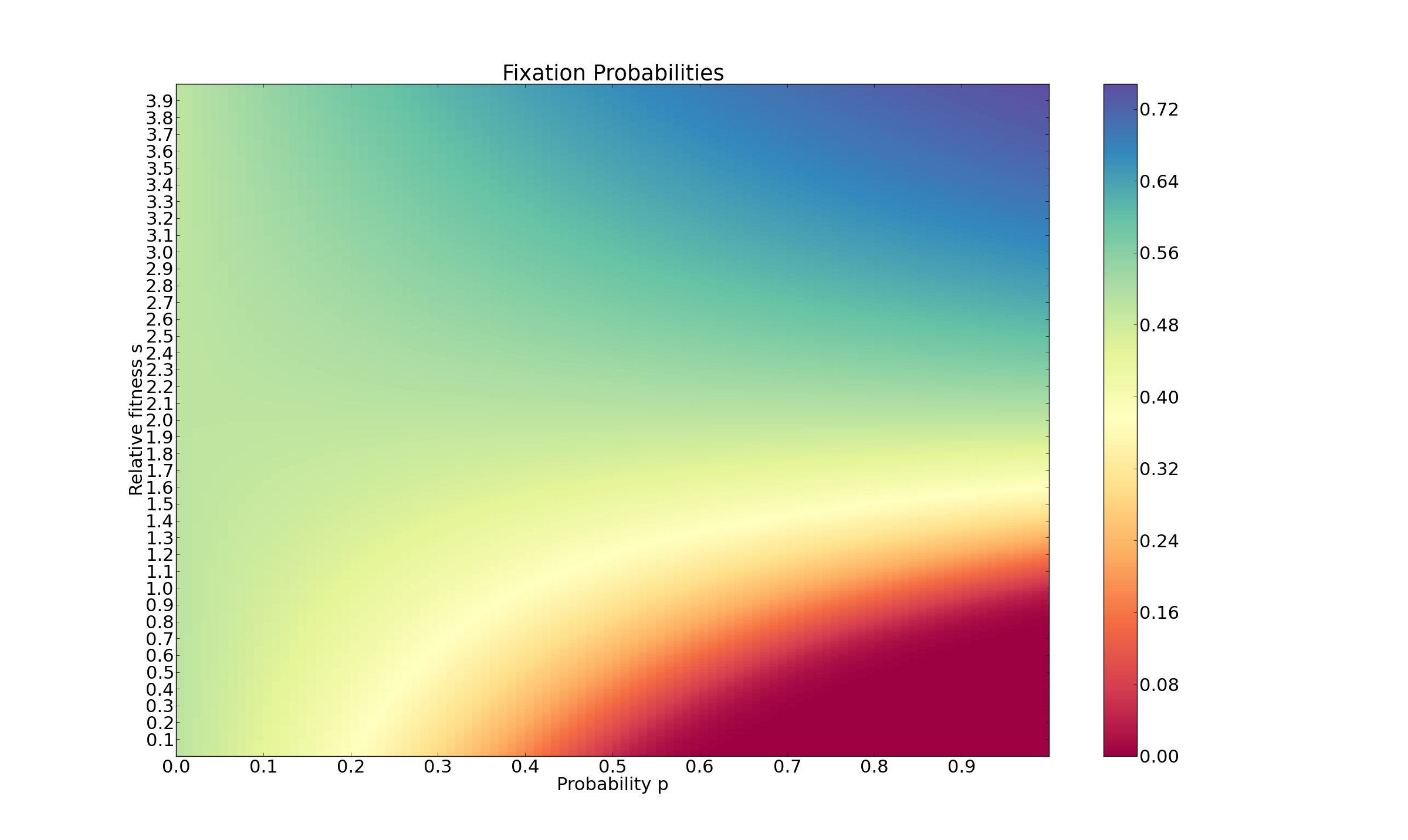}
    \caption{Fixation probabilities for a population of size $N=20$ and $t=1$. Along the horizontal axis runs the probability $p$; along the vertical runs the relative fitness $s$.}
\end{figure}

\subsection{Generational Processes}
Let us also briefly consider generational processes, such as the Wright-Fisher process \cite{imhof2006evolutionary} and the $n$-fold Moran process \cite{harper2013inherent}. The latter process is defined like the Moran processes, however each step of the process is defined not by a single replication event but rather by $N$-replication events, where $N$ is the population size. This does not change the fixation probabilities, except in the following special case. If $s=0$ and $t > 0$, when the population is subject to environment $E_1$, every fitness proportionate replication event will result in the reproduction of type $2$, so a fully generational process will fixate on type $2$. In other words, if the population must undergo an entire generation in one environment, it may strongly select the more adapted type than if the stochastic environmental changes are more fleeting.

A concrete example would be an organism that lives in an environment that periodically floods, say $E_1$, like certain amphibians. If water is required to reproduce successfully for one of the two types, say type $1$, $T_1$-individuals can only reproduce when environment $E_1$ is active (meaning $t=0$). A prolonged period (such as the time required for a full generation to pass) without flooding may lead to the extinction of type $1$ completely, even if type $1$ is far more effective than type $2$ at reproducing during flooding (i.e. $s >> 1$).

\section{Stochastic Incentive Switching}

Now we consider models which allow not only how the fitness landscape to change but also how the population interacts with the fitness landscape. We introduce the notion of an incentive, a functional parameter with the following properties that simultaneously generalizes common game dynamics like the replicator, the orthogonal-projection, the logit, and the best-reply dynamics. We briefly cover preliminaries here; see \cite{fryer2012existence} \cite{harper2012stability}. Table \ref{incentives_table} lists incentive functions for some common dynamics.  

Motivated by game-theoretic considerations, the incentive dynamics takes the form
\begin{equation}
\dot{x_i} = \incentive_i(x) - x_i \sum_{j}{\incentive_j(x)}.
\label{incentive_dynamic}
\end{equation}

\begin{figure}[h]
    \centering
    \begin{tabular}{|c|c|}
        \hline
        Dynamics & Incentive \\ \hline         \hline
        Replicator & $\incentive_i(x) = x_i \left(f_i(x) - \bar{f}(x)\right)$\\ \hline
        Best Reply & $\incentive_i(x) = BR_i(x) - x_i$ \\ \hline
        Logit & $\displaystyle{ \incentive_i(x) = \frac{\text{exp}(\eta^{-1}f_i(x))}{\sum_j{\text{exp}(\eta^{-1}f_j(x))}}}$ \\ \hline
        Projection & $\incentive_i(x) = \begin{cases}
                    & \displaystyle{f_i(x) - \frac{1}{|S(f(x), x)|}\sum_{j \in S(f(x), x)}{f_j(x)}} \quad \text{if $i \in S(f(x), x)$}\\
                    &0 \qquad \text{else}
                    \end{cases}$\\
        \hline
        \end{tabular}
    \caption{Incentives for some common dynamics, where $BR_i(x)$ is the best reply to state $x$, $S(f(x), x)$ is the set of all strategies in the support of $x$ as well as any collection of pure strategies in the complement of the support that maximize the average. Note that on the interior of the simplex the projection incentive is just $\incentive_i(x) = f_i(x) - 1/n \sum_j{f_j(x)}$. For more examples see Table 1 in \cite{fryer2012existence}.}
    \label{incentives_table}
\end{figure}

One interpretation of an incentive is that it mediates how the population interacts with the fitness landscape, though strictly-speaking an incentive requires no fitness landscape. The natural stability concept for incentive dynamics is the incentive stable state (ISS), and there is a natural generalization \cite{fryer2012kullback} of the well-known Lyapunov theorem \cite{akin1984evolutionary} for the replicator equation to incentive dynamics. An incentive stable state is a state $\ess{x}$ such that in a neighborhood of $\ess{x}$ the following inequality holds
\begin{equation}
 \sum_{i}{ \frac{\ess{x}_i \incentive_i(x)}{x_i}} > \sum_{i}{ \frac{x_i \incentive_i(x_i)}{x_i}}.
 \label{iss} 
\end{equation}

The KL-divergence on discrete probability distributions is a positive definite function defined \cite{kullback1951information} as
\[D_{KL}(x) = D_{KL}(\ess{x} || x) = \sum_i{\ess{x}_i \log {x_i} - \ess{x}_i \log \ess{x}_i}.\]

\begin{theorem*}[Fryer, 2011]
If the state $\ess{x}$ is an interior incentive stable state for the corresponding incentive dynamics, then $D_{KL}(x)$ is a local Lyapunov function.
\label{kl_divergence_incentives}
\end{theorem*}

The incentive $\incentive_i(x) = x_i f_i(x)$ captures the classical result for the replicator dynamics \cite{akin1984evolutionary} for all incentives locally, with the definition of ISS being exactly evolutionary stability (ESS): $\ess{x} \cdot f(x) > x \cdot f(x)$. For many common dynamics, such as the replicator, best reply, and orthogonal projection dynamics, the ISS condition is the same as ESS. This allows us to formulate the following theorem, which is an easy and direct consequence of the definition of incentive and Fryer's Theorem, and so we omit the proof.
\begin{theorem}
Let $\incentive$ and $\psi$ be incentives. Then
\begin{enumerate}
 \item the function $\incentive_p = p \incentive + (1-p) \psi$ is an incentive for all $p \in [0,1]$; and \\
 \item if $\ess{x}$ is an ISS for both $\incentive$ and $\psi$ then $\ess{x}$ is an ISS for $\incentive_p$ and the KL-divergence is a local Lyapunov function for the incentive dynamic defined by $\incentive_p$.
\end{enumerate}
\end{theorem}

Suppose that $f$ is a fitness landscape with $ESS$ $\ess{x}$ and let $\incentive$ and $\psi$ to be the incentives for the replicator and the best reply dynamics, respectively, associated with the landscape $f$. Then for all $p \in [0,1]$, $\incentive_p$ has ESS $\hat{x}$ and we have a local Lyapunov function from the KL-divergence. Similarly, we can combine the replicator incentive and the projection incentive (for interior trajectories of the simplex), and similarly obtain a local Lyapunov function for the dynamic. The function $D_0(x) = (1/2)||\hat{x} - x||^2$ is a global Lyapunov function for the projection dynamic \cite{lahkar2008projection} \cite{harper2011escort}. A straight-forward computation shows that $V = p D_{KL}(x) + (1-p) D_0(x)$ is a Lyapunov function for $\incentive_p = p \incentive + (1-p) \psi$ for all $p \in [0,1]$ (if $\ess{x}$ is ESS for both incentives).

We also give examples where the phase portrait qualitatively changes while keeping the landscape constant. Let $\incentive$ be the replicator incentive with the fitness landscape defined by the RSP matrix $A$ defined by $a=b=1$, which has concentric (non-attracting) cycles and a rest point at the barycenter of the simplex. If we take $\psi_i = \sum_{j=i}^{3}{(A_{ij} - x \cdot Ax)_{+}} = a = 1$, then the dynamic defined by $\incentive_p$ has phase portrait with asymptotically stable rest point at the barycenter (qualitatively as if $b > a$ for an RSP matrix). The dynamic associated to $\psi$ is $\dot{x}_i = a(1 - 3x_i)$, which has ISS at the barycenter \cite{fryer2012uniform}, and exponential trajectories from $x$ to $\hat{x}$ with rate $-3a$. Similarly, for a diverging RSP landscape, a constant incentive $\psi$ can cause the phase portrait to change to convergence to the barycenter, depending on the relative values of $p$, $a$, and $b$. For instance, $p=0.9$, $a=1$, $b=2$ is such an example. This shows that addition of an alternative incentive can turn an unstable equilibrium into an ISS. It is easy to see that the addition of a large enough constant incentive would do the same for all $p$.

\section{Discussion}

We have considered mean dynamics for several models for the fixation of traits for organisms that experience two distinct environments, encoded by changes in the fitness landscape or changes in the evolutionary incentive. For deterministic models, we have shown that the addition of a second landscape can alter the phase portrait of the population under one of the environments alone, splitting up the parameter space for the probability $p$ in various ways, depending on the landscapes used, and whether the model uses a symmetric or asymmetric landscape. We have shown that an alternative incentive may preserve the phase portrait or alter it, in the process defining and analyzing novel evolutionary dynamics. For an analogous stochastic model, we see that the fixation probabilities are altered by the switching landscapes, and that there are cases where fixation is unlikely under the stochastic model but inevitable under the deterministic model.

\subsection*{Methods}
Figure \ref{moran_fixation_probability} was created with \emph{matplotlib} \cite{Hunter:2007}.

\subsection*{Acknowledgments}
Marc Harper acknowledges support by the Office of Science (BER),
U. S. Department of Energy, Cooperative Agreement No. DE-FC02-02ER63421.

\bibliography{ref}
\bibliographystyle{plain}

\end{document}